\documentclass[12pt, reqno]{amsart}
\usepackage{color}

\usepackage{amsmath}
\usepackage{amssymb}
\usepackage{amsfonts}
\usepackage[lmargin=3cm,rmargin=3cm,tmargin=3cm,bmargin=3cm]{geometry}

\def\R{\mathbb{R}}
\def\N{\mathbb{N}}
\def\C{\mathbb{C}}

\numberwithin{equation}{section}

\newtheorem{theorem}{Theorem}[section]
\newtheorem{lemma}[theorem]{Lemma}
\newtheorem{proposition}[theorem]{Proposition}
\newtheorem{corollary}[theorem]{Corollary}

\newtheorem{definition}[theorem]{Definition}

\newtheorem{remark}[theorem]{Remark}

\author{ Abdelwahab Bensouilah}

\address{
 Laboratoire Paul Painlev\'e (U.M.R. CNRS 8524), U.F.R. de Math\'ematiques, Universit\'e Lille 1, 59655 Villeneuve d'Ascq Cedex, France
}
\email{\sl ai.bensouilah@math.univ-lille1.fr}

\author[D. Draouil]{Dhouha Draouil}
\address{Universit\'e de Tunis El Manar, Facult\'e des Sciences de Tunis, D\'epartement de Math\'ematiques, Laboratoire \'equations aux d\'eriv\'ees partielles (LR03ES04), 2092 Tunis, Tunisie}
\email{\sl douhadraouil@yahoo.fr }

\author{Mohamed Majdoub}
\address{Imam Abdulrahman Bin Faisal University, College of Science, Mathematics Department, Dammam, KSA}
\email{\sl mmajdoub@iau.edu.sa}

\title[Weighted critical NLS]
{Energy critical Schr\"odinger equation with weighted exponential nonlinearity \bf{I}: \\ Local and global well-posedness}

\date{\today}
\begin{document}
\begin{abstract}
We investigate the initial value problem for a defocusing nonlinear Schr\"odinger
equation with weighted exponential nonlinearity
$$
i\partial_t u+\Delta
u=\frac{u}{|x|^b}\big(e^{\alpha|u|^2}-1\big); \qquad (t,x) \in \R\times\R^2,
$$
where $0< b <1$ and $\alpha=2\pi(2-b)$. We establish local and global well-posedness in the subcritical
and critical regimes.
\end{abstract}


\maketitle
\tableofcontents
\vspace{ -1\baselineskip}
\eject

\section{Introduction}

This is the first of a two-paper series in which we deal with the initial value problem for a nonlinear Schr\"odinger
equation with weighted exponential nonlinearity
\begin{equation}
\label{eq1}
\left\{
\begin{matrix}
i\partial_t u+\Delta u= \omega(x)g(u), & u:(-T_*, T^*) \times \R^2 \longmapsto \C \\
u(0) = u_0 \in H^1 (\R^2) \\
\end{matrix}
\right.
\end{equation}
where
\begin{equation}
\label{eq2}
g(u)=u\big({\rm e}^{(4-2b)\pi|u|^2}-1\big),
\end{equation}
and the singular weight $\omega$ is given by
\begin{equation}
\label{eq3}
\omega(x)=|x|^{-b},
\end{equation}
with $b>0$.

Solutions of (\ref{eq1}) formally satisfy the conservation of mass and Hamiltonian
\begin{equation} \label{M} M(u(t)):= \| u(t)\|_{L^2}^2=M(u_0),
\end{equation}

 \begin{eqnarray}
 \nonumber
 H(u(t)):&=&  \int_{\R^2}|\nabla u(t)|^2\,dx+ \frac{1}{(4-2b)\pi} \int_{\R^2}\left( e^{(4-2b)\pi|u(t)|^2} -1-(4-2b)\pi |u(t)|^2\right)\frac{dx}{|x|^b}\\
 \label{H}&=& H(u_0)
 \end{eqnarray}

Before going further, let us recall a few historic facts about this problem. We begin with the case of a power defocusing nonlinear Schr\"{o}dinger equation
\begin{eqnarray} \label{MHIS}i\partial_t u+\Delta u = u | u|^{p-1},\quad 1<p<\infty.
\end{eqnarray}
A solution $u$ to \eqref{MHIS} satisfies conservation of mass and Hamiltonian
\begin{eqnarray} H(u(t)):= \| \nabla u(t)\|_{L^2}^2+\frac{2}{p+1}\int_{\R^N}| u|^{p+1}(t,x)dx.
\end{eqnarray}

This case has been widely investigated and there is a large literature dealing with the local and global solvability in the scale of the Sobolev spaces $H^s(\R^N)$ (see \cite{Liu,RS} for $N\geqslant3$ and \cite{Wang} for $N=2$). The case of weighted nonlinear terms $$f(x,u)=\lambda| x|^{-a} u | u|^{p};\qquad p,a>0\quad \textit{and} \quad  \lambda=\pm1,$$ was studied  in \cite{Caz-B,  Farah, Genoud}. \\

Note that, if $u$ solves \eqref{MHIS} then, for $\lambda>0$, $u^{\lambda}$ defined by
$$ u^{\lambda} (t,x):=\lambda^{\frac{2}{1-p}}u(\lambda^{-2}t,\lambda^{-1}x)$$
also solves \eqref{MHIS}. Let $s_c=\frac{d}{2}-\frac{2}{p-1}$. Remark that the $L^2$- based homogeneous $\dot{H}^{s_c}$-Sobolev norm is invariant under the mapping $f(x)\mapsto \lambda^{-\frac{2}{p-1}}f(\lambda^{-1}x)$ for all $\lambda>0$. It is known that if $s>s_c$, then \eqref{MHIS} is locally well-posed in $H^s(\R^N)$, with existence interval depending only upon $\| u_0\|_{H^s}$. For $s=s_c$, \eqref{MHIS} is locally well-posed in $H^s$, with an existence interval depending upon $e^{it\Delta}u_0$. Finally, if $s<s_c$, then \eqref{MHIS} is ill-posed in $H^s$. Therefore, it is natural to refer to $H^{s_c}$ as the critical regularity for \eqref{MHIS}.

For the energy critical case, i.e., $s_c=1$, an iteration of the local-in-time well posedness theory fails to prove global well posedness since the local existence interval does not depend only on $\| u_0\|_{H^1}.$ By using new ideas of Bourgain \cite{Bou} along with a new interaction Morawetz inequality \cite{CGT}, the energy critical case of \eqref{MHIS} is now completely solved (see \cite{Visan}).\\

For $N=2$, the initial value problem associated to \eqref{MHIS} is energy subcritical for all $p>1.$ To identify an "energy critical" nonlinear
Schr\"odinger initial value problem on $\R^2$, it is thus natural to consider problems with exponential nonlinearities. The exponential type nonlinearities appear in several applications, as for example the self trapped beams in plasma (See \cite{LLT}).
From a mathematical point of view, Cazenave  \cite{Caz2D} considered the nonlinearity $f(u)=(1-e^{-|u|^2})u$ and showed global well-posedness and scattering. We stress here that for exponential nonlinearities behaving like ${\rm e}^{|u|^2}$ at infinity, the nonlinear interaction grows more rapidly than any power for large amplitude, and moreover, the higher derivatives grows even faster. This is the main difficulty for the exponential nonlinearity. The two dimensional case is particularly interesting because of its relation to the critical Sobolev (or Moser-Trudinger) embedding.

The 2D-NLS problem with exponential nonlinearities was studied by Nakamura and Ozawa in \cite{NO} for small Cauchy data. They proved global well-posedness and scattering. Recently, Colliander-Ibrahim-Majdoub-Masmoudi \cite{CIMM} considered the following 2D-NLS initial value problem
\begin{equation*}
\left\{
\begin{matrix}
i \partial_t u + \Delta u =  u\left({\rm e}^{4\pi| u|^2}-1\right), \\
u(0) = u_0 \in H^1(\R^2).
\end{matrix}
\right.
\end{equation*}
They obtained global well-posedness for both subcritical and critical regimes (i.e., $H(u_0)\leqslant1$), and some ill-posedness results in the supercritical case ($H(u_0)>1$). Later on, Ibrahim-Majdoub-Masmoudi-Nakanishi proved in \cite{IMMN} the scattering  when the cubic term is subtracted from the nonlinearity to avoid another critical exponent related to the decay property of solutions. The main ingredient is a new interaction Morawetz estimate, proved independently by Colliander et al. \cite{CGT} and Planchon and Vega \cite{PV}. This estimate gives a priori global bound on $u$ in the space $L_t^4(L_x^8).$ Hence, by complex interpolation they deduced that some of the Strichartz norms used in the nonlinear estimate go to zero for large times and the scattering in the subcritical case follows.\\

Our main goal here is to prove global well-posedness for the Cauchy problem \eqref{eq1} in the energy space. The fundamental tools  are the classical Strichartz estimates and a new Moser-Trudinger inequality with singular weights proved recently in \cite{Souza1, Souza2, Souza3}.
\subsection{Statement of the results}

We begin by  defining our notion of criticality and
well-posedness for \eqref{eq1}. We then give precise statements of
our main results.

\begin{definition}
\label{d1} The Cauchy problem  (\ref{eq1})  is
said to be {\it subcritical} if
$$
H(u_0)<1.
$$
It is {\it critical} if $H(u_0)=1$ and
{\it supercritical} if $H(u_0)>1$.
\end{definition}

\begin{definition}
\label{d2}
 We say that the Cauchy problem (\ref{eq1}) is {\it locally well-posed} in $H^1(\R^2)$ if there
exist $E>0$ and a time $T=T(E)>0$ such that for every $u_0\in
B_E:=\{\;u_0\in H^1(\R^2);\quad \|\nabla u_0\|_{L^2}<E\;\}$ there
exists a unique (distributional) solution $u:
[-T,T]\times\R^2\longrightarrow\C$ to (\ref{eq1}) which is in the
space ${\mathcal C}([-T,T]; H^1_x)$, and such that the solution
map $u_0\longmapsto u$ is uniformly continuous from $B_E$ to
${\mathcal C}([-T,T]; H^1_x)$.
\end{definition}

 We give now our main result in the following global well-posedness theorem.
\begin{theorem}\label{EG}
Assume that $0< b <1$. Let $u_0\in H^1(\R^2)$ such that $H(u_0)\leqslant1$. Then, the Cauchy problem \eqref{eq1}  has a unique global solution $u\in C(\R,H^1(\R^2))$. Moreover, $u\in L_{loc}^4(\R,C^{\frac{1}{2}}(\R^2))$ and satisfies the conservation laws \eqref{M} and \eqref{H}.
\end{theorem}

In a forthcoming paper \cite{BDM2}, we will show that every global solution
of \eqref{eq1} with $H(u)< 1$ approaches solutions to the
associated free equation
\begin{equation}
\label{SL}
 i\partial_t v + \Delta v=0,
\end{equation}
in the energy space $H^1$ as $t\to\pm\infty$. So far, we have
not succeeded to handle the critical case $H(u)=1$, so we have to restrict
ourselves to the sub-critical one. More precisely, we have the following forward in-time scattering result.

\begin{theorem}\label{SS}
Assume that $0< b <1$. Then, for any global solution $u$ of \eqref{eq1} in $C(\R,H^1(\R^2))$ satisfying $H(u)<1$, we have $u\in L^4(\R,C^{\frac{1}{2}}(\R^2))$ and there exists a unique free solution $u_{+}$ to \eqref{SL} such that
\begin{eqnarray*}  \ \ \|(u-u_{+})(t)\|_{H^1}\underset{t\rightarrow +\infty}{\longrightarrow}0.\ \
\end{eqnarray*}
\end{theorem}

The paper is organized as follows. In Section 2, we present some useful tools from the literature that will be needed in our analysis. Section 3 is devoted to the proof of our main result about global well-posedness (Theorem \ref{EG}). We end the paper with an appendix in which we mainly focus on Schwarz symmetrization.\\

Here and below $C_T (X)$ denotes C([0, T ); X), $L_T^p (X)$ denotes $L^p([0, T );X)$ and $L^q$ denotes $L^q(\R^2)$.
If A and B are nonnegative quantities, we use $A \lesssim B$ to denote $A \leqslant CB$ for
some positive universal constant C. Finally, B(R) denotes the ball in $\R^2$ centered at the origin and with radius R.

\section{Basic facts and useful tools}
In this section we introduce the basic tools we will use all along this paper. We start with the following weighted Moser-Trudinger type inequality.

\begin{theorem}\cite{Souza1}\label{MT}
Let $0< b <2$ and $0<\alpha<2\pi(2-b)$. Then, there exists a positive constant $C=C(b,\alpha)$ such that
\begin{equation}
\label{WMT1}
\int_{\R^2}\frac{e^{\alpha| u(x)|^2}-1}{| x|^b}dx\leqslant C\int_{\R^2}\frac{| u(x)|^2}{| x|^b}dx,
\end{equation}
for all $u\in H^1(\R^2)$ with $\| \nabla u\|_{L^2(\R^2)}\leqslant 1$.
\end{theorem}
We point out that $\alpha=2\pi(2-b)$ becomes admissible in
(\ref{WMT1}) if we require $\|u\|_{H^1(\R^2)}\leq1$ instead of
$\|\nabla u\|_{L^2(\R^2)}\leq1$. More Precisely, we have
\begin{theorem}\cite{Souza2}\label{MTT}
Let $0< b <2$. We have
\begin{equation}
\label{WMT2}
\sup_{\| u\|_{H^1(\R^2)}\leqslant 1}\int_{\R^2}\frac{e^{\alpha| u(x)|^2}-1}{| x|^b}dx<\infty\quad\mbox{if and only if}\quad \alpha\leq 2\pi(2-b).
\end{equation}
\end{theorem}
The following lemma will be very useful.
\begin{lemma}
\label{hardy}
Let  $0< b <2$ and $\gamma \geq 2$. Then, there exists a positive constant $C=C(b,\gamma)>0$ such that
\begin{equation}
\label{Hardy}
\int_{\R^2}\frac{|u(x)|^{\gamma}}{|x|^b} dx\leq\,C \|u\|_{H^1(\R^2)}^{\gamma},
\end{equation}
for all $u\in H^1(\R^2)$.
\end{lemma}
\begin{remark}
Inequality \eqref{Hardy} fails for $b\geqslant2$. Indeed, let $u\in \mathcal{D}(\R^2)$ (the space of smooth compactly supported functions) be a radial function such that
$u(x)\equiv1$ for $| x| \leqslant 1$. Then, $u\in H^1(\R^2)$ and
$$\int_{\R^2}\frac{| u(x)|^\gamma}{| x|^b} dx\geqslant 2\pi \int_0^1\frac{rdr}{r^b}=+\infty.$$
\end{remark}
\begin{proof}[{Proof of Lemma \ref{hardy}}]
The diamagnetic inequality (see \cite{Tao} for a proof)
$$
\int_{\R^2} |\nabla |u||^2 \, dx \leqslant \int_{\R^2} |\nabla u|^2 \, dx,
$$
which holds true for all functions  $u\in H^1(\R^2)$, shows that it suffices to prove inequality \eqref{hardy} for $u\in H^1(\R^2)$ such that $u\geqslant0$ on $\R^2$.\\
Let $u$ be a non-negative $H^1$-function. Write
\begin{eqnarray*}
\int_{\R^2}\frac{| u(x)|^\gamma}{| x|^b} dx &=&\int_{\{| x| \leqslant1 \}}\frac{| u(x)|^\gamma}{| x|^b} dx+\int_{\{| x|>1\}} \frac{| u(x)|^\gamma}{| x|^b} dx\\&=& I+II.
\end{eqnarray*}
 The Sobolev embedding $H^1(\R^2)\hookrightarrow L^\gamma(\R^2)$ implies that
 \begin{equation}
 \label{First}
 II \leqslant \| u\|_{L^\gamma}^\gamma\leqslant C_\gamma\| u\|_{H^1}^\gamma.
 \end{equation}
Set $\Omega:=\{| x|<2\}$ and $v:\Omega \rightarrow \R$, \, $x \mapsto u(x) \phi(x)$, where $\phi$ is a $C^1_c(\Omega)$-function such that $\phi\equiv 1$ on $\{| x|\leqslant 1\}$, \, $0\leqslant \phi \leqslant 1$ and $\phi(x)= 0$ for $\frac{3}{2} < |x| <2$. We note that $v \in H^1_0(\Omega)$. Thus, the Polya-Szeg\"{o} theorem implies that $v^* \in H^1_0(\Omega^*)$. Here, $\psi^*$ stands for the Schwarz symmetrization of $\psi$.\\
We have
$$
I \leqslant\int_{\Omega}\frac{| v(x)|^\gamma}{| x|^b} dx=\int_{\Omega}\frac{ v(x)^\gamma}{| x|^b} dx,
$$
where in the last inequality we used the fact that $v$ is a non-negative function.\\
On the one hand, since $0<b<2$, one can find $1\leqslant q < \infty$ such that $0<b<\frac{2}{q'}$, so that $\omega \in L^{q'}(\Omega)$. On the other hand, by the Sobolev embedding, we know that $v^{\gamma} \in L^q(\Omega)$.  Hence, we can apply the Hardy-Littlewood inequality to get
$$
\int_{\Omega}\frac{ v(x)^\gamma}{| x|^b} dx \leqslant \int_{\Omega^*} (F(v))^*(x) \omega^*(x) dx,
$$
where $F$ denotes the function
$$
F(y):=y^{\gamma} \chi_{\R^{+}}(y), \qquad y\in \R.
$$
The fact that $F$ is non-decreasing gives
$$
(F(v))^*=F(v^*) \leqslant |v^*|^{\gamma}.
$$
A simple computation shows that
$$
\omega^{\#}(s)=\pi^{\frac{b}{2}} s^{-\frac{b}{2}}, \quad s \in ]0,4 \pi],
$$
where $\omega^{\#}$ stands for the decreasing rearrangement of $\omega$. Thus
$$
\omega^*(x):=\omega^{\#}(\pi |x|^2)=\omega(x), \quad \textit{a.e.} \quad x \in \Omega^*.
$$
Therefore,
$$
\int_{\Omega}\frac{ v(x)^\gamma}{| x|^b} dx \leqslant \int_{\Omega^*} \frac{| v^{*}(x)|^\gamma}{|x|^{b}} dx= 2 \pi\int_0^2 \frac{| v^{*}(r)|^\gamma}{r^{b-1}}dr.
$$
Extending $v^*$ by zero outside its domain, we can consider it as an element of $H^1_{\textit{rad}}(\R^2)$. The Strauss radial lemma (see appendix) and the Poincar\'e inequality imply then
$$
I\leqslant 2\pi (C_{p})^\gamma \| v^{*}\|_{H^1_0(\Omega^*)}^\gamma\int_0^2\frac{dr}{r^{\frac{2\gamma}{p+2}+b-1}}.
$$
Choose $2\leqslant p<\infty$ such that
$$\frac{2\gamma}{p+2}+b-1<1.$$
This is possible since $b<2$. Hence,  Polya-Szeg\"{o} theorem and the Poincar\'e inequality yield
$$I\leqslant C(p,\gamma,b) \| v\|_{H^1(\Omega)}^\gamma.$$
Since $\| v\|_{H^1(\Omega)} \lesssim \| u\|_{H^1(\R^2)}$, one obtains at final
\begin{equation}
\label{Second}
I \lesssim \| u\|_{H^1(\R^2)}^\gamma.
\end{equation}
Combining inequalities \eqref{First} and \eqref{Second} we deduce the desired result.
\end{proof}

In order to control solutions of $(1.1)$, we will use Strichartz estimates.
\begin{proposition}[\sf Strichartz estimates \cite{CIMM}]\quad\\
Let $v_0$ be a function in $H^1(\R^2)$
and $F \in L^1(\R,H^1(\R^2))$. Denote by v a solution to the inhomogeneous linear
Schr$\ddot{o}$dinger equation
$$i\partial_t v +\Delta v = F,$$
with initial data $v(0, x) = v_0(x)$.
Then, a constant C exists such that for any $T > 0$ and any admissible couple $(q,r),(\delta,\rho)$ we have
\begin{eqnarray}\label{ST} \| v\|_{L^q([0,T ],W^{1,r}(\R^2))}\leqslant C
\left[\| v_0\|_{H^1(\R^2)} +
\| F\|_{L^{\delta'}([0,T ],W^{1,{\rho'}})(\R^2))}\right].\end{eqnarray}
We recall that a couple (q,r) is said to be $Schr\ddot{o}dinger$ admissible, if $2\leqslant q,r\leqslant\infty,\ (q,r)\neq(2,\infty)$ and $\frac{1}{q}+\frac{1}{r}=\frac{1}{2}.$
\end{proposition}

Thanks to the next inequality, we will be able to control the exponential term.
\begin{proposition}[\sf Log Estimate \cite{PAMS}]\quad\\
Let $\beta\in ]0,1[.$ For any $\lambda > \frac{1}{2\pi\beta}$ and any
$0<\mu\leqslant 1$, a constant $C_{\lambda} > 0$ exists such that, for any function $u \in H^1(\R^2)\bigcap C^{\beta}(\R^2)$,
we have
\begin{eqnarray}\label{LE}\| u\|_{L^{\infty}(\R^2)}^2 \leqslant \lambda \| u\|_{H_{\mu}}^2 log\bigg(C_{\lambda}+\frac{8^{\beta}\mu^{-\beta}
\| u\|_{C^{\beta}(\R^2)}}{\| u\|_{H_{\mu}}}\bigg),\end{eqnarray}
where we set
$$\| u\|_{H_{\mu}}^2:= \|\nabla u\|_{L^2(\R^2)}^2 +\mu^2 \| u\|_{L^2(\R^2)}^2 .$$
Recall that $C^{\beta}(\R^2)$ denotes the space of $\beta$-$h\ddot{o}lder$ continuous functions endowed
with the norm
$$\| u\|_{C^{\beta}(\R^2)}:= \| u\|_{L^{\infty}(\R^2)}+ \sup_{x\neq y} \frac{| u(x)-u(y)|}{| x-y|^{\beta}}.$$
We note that $(q,r)=(4,4)$ is an admissible Strichartz couple and we have $$W^{1,4}(\R^2)\hookrightarrow C^{\frac{1}{2}}(\R^2).$$
\end{proposition}
We end this section with a lemma that will be useful in the proof of the global existence.

\begin{lemma}\label{LG}
Let  $u$ be a solution of \eqref{eq1} on $[0,T)$ with $0<T\leqslant \infty$ and suppose that $E:=H(u_0)+M(u_0) < \infty$.
Then, a constant $C(E)$ exists such that, for any two positive real numbers $S$ and $S'$ and for any $0<t<T$, the following holds
$$ \int_{B(S+S')} | u(t,x)|^2 dx \geqslant \int_{B(S)} | u_0(x) | ^2 dx -  \frac{C(E)}{S'}t.$$
\end{lemma}
\begin{proof}
Let $h:\R\rightarrow\R$  be a real-valued $C^{\infty}$-function such that
$$\left \{\begin{array}{l}
h(\tau)=1\ \ \  if\ \ \  \tau\geqslant1,\\
h(\tau)=0\ \ \  if\ \ \ \tau\leqslant0,\\
0\leqslant h(\tau)\leqslant1 \ \ \ \forall \tau\in \R.\end {array}
 \right.$$\\

Set $$\psi(x):= h\left(1-\frac{d_S(x)}{S'}\right),$$
where $$d_S(x):=\left \{\begin{array}{l}
0\ \ \  if\ \ \  | x|\leqslant S\\
| x| -S\ \ \  if\ \ \  | x| \geqslant S.\end {array}
 \right.$$
Note that $\psi$ satisfies  $$\left \{\begin{array}{l}
\psi(x)=1\ \ \  if\ \ \  x\in B(S),\\
\psi(x)=0\ \ \  if \ \ \ | x|\geqslant S+S',\\
\| \nabla\psi\|_{L^\infty}\leqslant\frac{\| h'\|_{L^{\infty}}}{S'}.
\end {array}
 \right.$$

Now, Multiply equation \eqref{eq1} by $\psi^2\bar{u}$. We get
\begin{eqnarray}\label{ima} i(\psi^2\bar{u})\partial_tu+\psi^2\bar{u}\Delta u=\psi^2\bar{u} \omega(x)g(u) .\end{eqnarray}
Taking the complex conjugate of \eqref{eq1} and multiplying it by $\psi^2u$ gives
\begin{eqnarray}\label{imag} -i\psi^2u\partial_t\bar{u}+\psi^2u\Delta\bar{ u}=\psi^2u \overline{\omega(x)g(u)}.\end{eqnarray}
Subtracting \eqref{imag} from \eqref{ima} yields
$$i\psi^2(\bar{u}\partial_t u+u\partial_t\bar{u})+\psi^2(\bar{u}\Delta u-\Delta\bar{u} u)=0.$$
Integrating the last equation over $\R^2$ and then integrating by parts give
$$\partial_t\| \psi u\|_{L^2}^2=-4Im\left(\int_{\R^2}\psi(x) u\nabla\psi(x) \nabla\bar{u}dx\right).$$
We have
 \begin{eqnarray*}Im\left(\int_{\R^2}\psi(x)\nabla\psi(x) u\nabla\bar{u}dx\right)&\leqslant&\int_{B(S+S')}|\psi(x)\nabla\psi(x) u\nabla\bar{u}| dx \\&\leqslant&
 \frac{\| h'\|_{L^\infty}}{S'}\int_{\R^2}| u(x)| | \nabla u(x)| dx\\
 &\leq&\frac{ \| h'\|_{L^\infty}}{2 S'} E.
\end{eqnarray*}
Therefore
$$\partial_t\|\psi u\|_{L^2}^2\geqslant-\frac{C(E)}{S'}.$$
Let $0<t<T$ and integrate over $]0,t[$ to get
\begin{eqnarray*}
\int_{B(S+S')}| u(t,x)|^2dx&\geqslant& \int_{B(S+S')}\psi^2(x) | u(t,x)|^2dx\\&\geqslant&
\int_{B(S+S')}\psi^2(x)| u_0(x)|^2dx-\frac{C(E)}{S'}\,t\\
&\geqslant&\int_{B(S)}| u_0(x)|^2dx-\frac{C(E)}{S'}\,t.
\end{eqnarray*}
Here we have used the fact that $0\leq \psi\leq 1$ and $\psi \equiv 1$ on $B(S)$.
\end{proof}

\section{Local and global well-posedness}
This section is devoted to prove the existence and the uniqueness of solutions in the energy space to the problem  $$\left \{\begin{array}{l}
i\partial_t u+\Delta u =\frac{u}{| x|^b} (e^{\alpha| u|^2 }-1) ,\\
u(t=0,.) = u_0(.) .
\end {array}
 \right.\eqno{(P)}$$
Here $u_0\in H^1(\R^2)$, $\alpha:= (4-2b)\pi$ and $0< b<1 $. First, we construct a local in-time solution using a classical fixed point argument in a suitable complete metric space. Then, we show the uniqueness in its unconditional form. Finally, using a non-concentration argument we prove that the maximal solution is global in both subcritical and critical regimes. For further purposes, we set

\begin{equation}
\label{phi}
\Phi(u)(t):={\rm e}^{it\Delta}u_0+i\int_{0}^t\,{\rm e}^{i(t-\tau)\Delta} f(x, u(\tau,x))\,d\tau,
\end{equation}
where
\begin{equation}
\label{nonlin}
f(x,u):=\omega(x)g(u).
\end{equation}

Let us also introduce, for any non-negative time $T$, the following Banach space

$$E(T) = C([0,T], H^1(\R^2)) \cap L^4 ([0,T], C^{\frac{1}{2}}(\R^2))$$
 endowed with the norm
$$\| u\|_T:= \sup_{t\in[0,T]} \left(\| u(t)\|_{L^2(\R^2)} +\| \nabla u(t)\|_{L^2(\R^2)} \right)+ \| u\|_{L_T^4([0,T],C^{\frac{1}{2}}(\R^2))}.$$
\subsection{Local existence}

We summarize the result in the following theorem.
\begin{theorem}
\label{local}
Let $u_0\in H^1(\R^2)$ such that $\| \nabla u_0\|_{L^2}<1$ and $0< b<1$. Then, there exist $T=T(\|u_0\|_{H^1(\R^2)},b)>0$ and a solution $u$ to $(P)$ in the class $E(T)$.
\end{theorem}
The proof of this theorem is carried out by a fixed point argument  combined with the following estimates.
\begin{lemma}
\label{tech-Loc}
Let $f$ be given by \eqref{nonlin}. Then, for any $\varepsilon>0$ there exists a constant $C_{\varepsilon, b }>0$ such that
\begin{equation}
\label{Loc1}
|f(x,u)-f(x,v)|\leq C_{\varepsilon, b } \frac{|u-v|}{|x|^b}\bigg({\rm e}^{\alpha(1+\varepsilon)|u|^2}-1+ {\rm e}^{\alpha(1+\varepsilon)|v|^2}-1\bigg)
\end{equation}
and
\begin{eqnarray}
\label{Loc2}
|\nabla \left(f(x,u)-f(x,v)\right)|&\leq& C_{\varepsilon, b } \bigg\{ \frac{|u-v|}{|x|^{b+1}}\bigg({\rm e}^{\alpha(1+\varepsilon)|u|^2}-1+ {\rm e}^{\alpha(1+\varepsilon)|v|^2}-1\bigg)\\
\nonumber &+&   \frac{|\nabla u-\nabla v|}{|x|^{b}}\left({\rm e}^{\alpha(1+\varepsilon)|u|^2}-1\right)\\
\nonumber &+&  |\nabla v| \frac{|u-v|}{|x|^{b}}\bigg(|u|+{\rm e}^{\alpha(1+\varepsilon)|u|^2}-1+ |v|+{\rm e}^{\alpha(1+\varepsilon)|v|^2}-1\bigg) \bigg\}
\end{eqnarray}
\end{lemma}

\begin{proof}
Let us identify $g$ with the ${\mathcal C}^\infty$-function defined on $\R^2$ and
denote by $D g$ the $\R^2$-derivative of the identified function. Using the
mean value theorem and the convexity of the exponential function, we derive the
following  properties:
\begin{eqnarray}
\nonumber |g(z_1)-g(z_2)|&\lesssim&
|z_1-z_2|\sum_{j=1,2}\left(e^{\alpha|z_j|^2}-1 +|z_j|^2
e^{\alpha|z_j|^2}\right),
\end{eqnarray}
and

\begin{eqnarray}
\nonumber |(Dg)(z_1)-(Dg)(z_2)|\lesssim
|z_1-z_2|\sum_{j=1,2}\left(|z_j|e^{\alpha|z_j|^2} +
|z_j|^3e^{\alpha|z_j|^2}\right).
\end{eqnarray}

Therefore,  for any positive real number $\varepsilon$, there
exists a positive constant $C_{\varepsilon}$ such that
\begin{eqnarray}
\label{eq12} |g(z_1)-g(z_2)|\leq C_{\varepsilon} |z_1-z_2|
\Big\{e^{\alpha(1+\varepsilon)|z_1|^2}-1+
e^{\alpha(1+\varepsilon)|z_2|^2}-1\Big\},
\end{eqnarray}
and
\begin{eqnarray}
\label{eq13} |(Dg)(z_1)-(Dg)(z_2)|\leq C_{\varepsilon}  |z_1-z_2|
\sum_{i=1,2}\left(|z_i|+ e^{\alpha(1+\varepsilon)|z_i|^2}-1\right).
\end{eqnarray}

Let $u,v:\R\times\R^2\rightarrow \C$ be two complex-valued functions. A straightforward calculation gives
\begin{eqnarray}
\label{56}
\nonumber
\nabla \left(f(x,u)-f(x,v)\right)&=&\frac{1}{|x|^b} \left(Du-Dv\right) \cdot Dg(u)+\frac{1}{|x|^b} Dv \cdot \left(Dg(u)-Dg(v)\right)\\
&-&\frac{bx}{|x|^{b+2}} \left(g(u)-g(v)\right)
\end{eqnarray}
where
$$
(Dg)(\psi):=\left( \begin{array}{c}
e^{\alpha|\psi|^2}-1+ \alpha  |\psi|^2 e^{\alpha|\psi|^2}\\
\alpha \psi^2 e^{\alpha|\psi|^2}\\
\end{array} \right),
$$
and
$$
D\psi:=\left( \begin{array}{c}
\nabla \psi\\
 \nabla \bar{\psi}\\
\end{array} \right).
$$

Using estimates \eqref{eq12} and \eqref{eq13} with the fact that
$$
xe^x\leqslant \frac{e^{(1+\epsilon)x}-1}{\epsilon},
$$ for all $x\geqslant0$ and all $\epsilon>0$, we get inequality \eqref{Loc2}.

\end{proof}
 \begin{proof}(of theorem \ref{local})

For $T, R>0$, denote by $B_T(R)$ the ball in $E(T)$ of radius $R$ and centered at the origin. We define the map $\Phi$ on the ball $B_T(R)$ by $v\mapsto \Phi(v):=\tilde{v},$ where $\tilde{v}$ solves
 $$\left \{\begin{array}{l}
i\partial_t \tilde{v}+\Delta \tilde{v} = f(x,v+v_0),\\
\tilde{v}(t=0,x) = 0.
\end {array}
 \right.$$
  Here $v_0:= e^{it\triangle}u_0$ is the solution of the free $Schr\ddot{o}dinger$ equation \eqref{SL} with initial data $u_0$.
  In the sequel, we will prove that for $T>0$ and $R>0$ suitably chosen, $\Phi$ is a contraction map from $B_T(R)$ into itself.\\

We start by showing the stability of $B_T(R)$ by $\Phi$. Applying the Strichartz estimate \eqref{ST} one gets
$$\| \Phi(v)\|_{T}=\| \tilde{v}\|_{T}\lesssim \| f(x,v+v_0)\|_{L_T^{\frac{2(1+\varepsilon)}{3\varepsilon+1}}(W^{1,1+\varepsilon}_x)},$$
where $\varepsilon$ is a nonnegative real number to be chosen suitably.\\
Let us first estimate $\| f(x,v+v_0)\|_{L_T^{\frac{2(1+\varepsilon)}{1+3\varepsilon}}(L^{1+\varepsilon}_x)}$. H\"older inequality in space and time yields
$$\| f(x,v+v_0)\|_{L_T^{\frac{2(1+\varepsilon)}{1+3\varepsilon}}(L^{1+\varepsilon}_x)}\lesssim \bigg\| \frac{v+v_0}{| x|^{\frac{b}{2}}}\bigg\|_{L_T^\infty(L^2_x)} \bigg\| \frac{e^{\alpha\arrowvert v+v_0\arrowvert^2}-1}{| x|^{\frac{b}{2}}}\bigg\|_{L_T^{\frac{2(1+\varepsilon)}{1+3\varepsilon}}(L^{\frac{2(1+\varepsilon)}{1-\varepsilon}}_x)}.$$
From lemma \ref{hardy} one deduces that
$$\bigg\| \frac{v+v_0}{| x|^{\frac{b}{2}}}\bigg\|_{L_T^{\infty}(L^2_x)}\lesssim \| v+v_0\|_{L_T^{\infty}(H^1_x)}.$$
 Now, we will deal with the term
 $\bigg\| \frac{e^{\alpha|v+v_0|^2}-1}{| x|^{\frac{b}{2}}}\bigg\|_{L_T^{\frac{2(1+\varepsilon)}{1+3\varepsilon}}(L^{\frac{2(1+\varepsilon)}{1-\varepsilon}}_x)}.$ An easy computation gives
 \begin{equation}
 \bigg\| \frac{e^{\alpha\arrowvert v+v_0\arrowvert^2}-1}{| x|^{\frac{b}{2}}}\bigg\|_{L_T^{\frac{2(1+\varepsilon)}{1+3\varepsilon}}(L^{\frac{2(1+\varepsilon)}{1-\varepsilon}}_x)} \leqslant \bigg\|e^{ \alpha \Arrowvert v+v_0\Arrowvert_{L^{\infty}}^2}-1\bigg\|_{L^1_T}^{\frac{1+3\varepsilon}{2(1+\varepsilon)}} \bigg\|\frac{e^{\alpha| v+v_0|^2}-1}{| x|^{\frac{b(1+\varepsilon)}{1-\varepsilon}}}\bigg\|_{L^{\infty}_{T}L^1_x}^{\frac{1-\varepsilon}{2(1+\varepsilon)}}
 \end{equation}

 Since $b<2$, one can find $0<\varepsilon<\frac{2-b}{2+b} $. On the other hand, since $\|\nabla u_0\|_{L^2}<1$, one can find $0<R<\frac{1-\|\nabla u_0\|_{L^2}}{2}$, so that $\|\nabla( v+v_0)\|_{L^2}\leqslant \|\nabla v\|_{L^2}+\| \nabla v_0\|_{L^2}\leqslant \frac{1+\|\nabla u_0\|_{L^2}}{2} <1 $. Set $A:=\frac{1+\|\nabla u_0\|_{L^2}}{2} $. Applying successively theorem \ref{MT} and lemma \ref{hardy} one obtains the following estimate
 $$\bigg\| \frac{ e^{\alpha| v+v_0|^2}-1}{| x|^{\frac{b(1+\varepsilon)}{1-\varepsilon}}}\bigg\|_{L^{\infty}_{T}L^1_x}^{\frac{1-\varepsilon}{2(1+\varepsilon)}}\lesssim \| v+v_0\|_{T}^{\frac{1-\varepsilon}{1+\varepsilon}}.$$

 In order to estimate $\bigg\|e^{ \alpha \Arrowvert v+v_0\Arrowvert_{L^{\infty}_x}^2}-1\bigg\|_{L^1_T}$, we use the same technique as in \cite{BDM1}. Write
\begin{eqnarray*}
\bigg\|e^{ \alpha \Arrowvert v+v_0\Arrowvert_{L^{\infty}_x}^2}-1\bigg\|_{L^1_T}&=& \int_{ \{t\in [0,T]\, ; \, \Arrowvert v+v_0\Arrowvert_{L^{\infty}_x} \leqslant 1\}} \left(e^{ \alpha \Arrowvert v+v_0\Arrowvert_{L^{\infty}_x}^2}-1\right) dt\\
&+&\int_{ \{t\in [0,T]\, ; \, \Arrowvert v+v_0\Arrowvert_{L^{\infty}_x}  > 1\}} \left(e^{ \alpha \Arrowvert v+v_0\Arrowvert_{L^{\infty}_x}^2}-1\right)dt.
\end{eqnarray*}

 It can easily be shown that
$$
\int_{ \{t\in  [0,T]\, ; \, \Arrowvert v+v_0\Arrowvert_{L^{\infty}_x} \leqslant 1\}}\left(e^{ \alpha \Arrowvert v+v_0\Arrowvert_{L^{\infty}_x}^2}-1\right) dt \leqslant C(b) T^{\frac{1}{2}} \Arrowvert v+v_0\Arrowvert_{T}^{2}.
$$
 The log estimate allows us to find a constant $0<\gamma<1$ such that
 $$
\int_{ \{t\in [0,T]\, ; \, \Arrowvert v+v_0\Arrowvert_{L^{\infty}_x} > 1\}} \left(e^{ \alpha \Arrowvert v+v_0\Arrowvert_{L^{\infty}_x}^2}-1\right) dt \lesssim  \int_{ \{t\in [0,T]\, ; \, \| u_{\omega}(t,\cdot)\|_{L^{\infty}_x} > 1\}} \Arrowvert v+v_0\Arrowvert_{C^{\frac{1}{2}}_x}^{2 (2-b) \gamma}.
$$
Therefore
$$
\int_{ \{t\in [0,T]\, ; \, \Arrowvert v+v_0\Arrowvert_{L^{\infty}_x}  > 1\}} \left(e^{ \alpha \Arrowvert v+v_0\Arrowvert_{L^{\infty}_x}^2}-1\right) dt \lesssim T^{\frac{b}{2}} \Arrowvert v+v_0\Arrowvert_{T}^{2(2-b)}.
$$
 Thus
\begin{equation}
\label{estimate1}
 \| f(x,v+v_0)\|_{L_T^{\frac{2(1+\varepsilon)}{3\varepsilon+1}}(L^{1+\varepsilon}_x)} \lesssim \Arrowvert v+v_0\Arrowvert_{T}^{\frac{2}{1+\varepsilon}} \left(T^{\frac{1}{2}} \Arrowvert v+v_0\Arrowvert_{T}^{2}+T^{\frac{b}{2}} \Arrowvert v+v_0\Arrowvert_{T}^{2(2-b)}\right)^{\frac{1+3\varepsilon}{2(1+\varepsilon)}}.
 \end{equation}

Now we are going to estimate the gradient term, namely $\| \nabla f(x,v+v_0)\|_{L_T^{\frac{2(1+\varepsilon)}{3\varepsilon+1}}(L^{1+\varepsilon}_x)}$.
We have
 \begin{eqnarray*}
\| \nabla f(x,v+v_0)\|_{L^{\frac{2(1+\varepsilon)}{1+3\varepsilon}}_T(L^{1+\varepsilon}_x)}&\lesssim&
\bigg\| \frac{|v+v_0|}{| x|^{b+1}}(e^{\alpha| v+v_0|^2} -1)\bigg\|_{L^{\frac{2(1+\varepsilon)}{1+3\varepsilon}}_T(L^{1+\varepsilon}_x)}\\&+&
\bigg\| \frac{|\nabla(v+v_0)|}{| x|^{b}}(e^{\alpha(1+\varepsilon)| v+v_0|^2} -1)\bigg\|_{L^{\frac{2(1+\varepsilon)}{1+3\varepsilon}}_T(L^{1+\varepsilon}_x)}.
\end{eqnarray*}
In the sequel we set
$$
{\mathbf I}_1:= \bigg\| \frac {|v+v_0|}{| x|^{b+1}}(e^{\alpha| v+v_0|^2}-1)\bigg\|_{L_T^{\frac{2(1+\varepsilon)}{1+3\varepsilon}}(L^{1+\varepsilon}_x)},
$$
and
$$
{\mathbf I}_2:=\bigg\| \frac{|\nabla(v+v_0)|}{| x|^{b}}(e^{\alpha(1+\varepsilon)| v+v_0|^2} -1)\bigg\|_{L^{\frac{2(1+\varepsilon)}{1+3\varepsilon}}_T(L^{1+\varepsilon}_x)}.
$$
We have
\begin{eqnarray*}
{\mathbf I}_1&\leqslant&
 \bigg\|\frac {v+v_0}{| x|^{\frac{b+1}{2}}}\bigg\|_{L_T^{\infty}(L^2_x)}\bigg\|\frac{e^{\alpha| v+v_0|^2}-1}{| x|^{\frac{b+1}{2}}}\bigg\|_{L_T^{\frac{2(1+\varepsilon)}{1+3\varepsilon}}(L^{\frac{2(1+\varepsilon)}{1-\varepsilon}}_x)}.
 \end{eqnarray*}
The fact that $0<b<1$ allows us to apply successively theorem \ref{MT} and lemma \ref{Hardy} and to infer that
$$
\bigg\|\frac {v+v_0}{| x|^{\frac{b+1}{2}}}\bigg\|_{L_T^{\infty}(L^2_x)} \lesssim \Arrowvert v+v_0\Arrowvert_{T}.
$$
On the other hand,  from H\"{o}lder inequality in space and time we have
$$
\bigg\|\frac{e^{\alpha| v+v_0|^2}-1}{| x|^{\frac{b+1}{2}}}\bigg\|_{L_T^{\frac{2(1+\varepsilon)}{1+3\varepsilon}}(L^{\frac{2(1+\varepsilon)}{1-\varepsilon}}_x)} \leqslant \bigg\|\frac{e^{\alpha| v+v_0|^2}-1}{| x|^{\frac{(b+1)(1+\varepsilon)}{1-\varepsilon}}}\bigg\|_{L_T^{\infty}(L^1_x)}^{\frac{1-\varepsilon}{2(1+\varepsilon)}}\bigg\| e^{\alpha\| v+v_0\|_{L^{\infty}_x}^{2}}-1\bigg\|_{L_T^{1}}^{\frac{1+3\varepsilon}{2(1+\varepsilon)}}.
$$
Taking $0<\varepsilon<\frac{1-b}{b+3}$ ensures that $\frac{(b+1)(1+\varepsilon)}{1-\varepsilon}<2$. On the other hand, for $0<R<\frac{1-\|\nabla u_0\|_{L^2}}{2}$ one has $\|\nabla( v+v_0)\|_{L^2}\leqslant \|\nabla v\|_{L^2}+\| \nabla v_0\|_{L^2}\leqslant \frac{1+\|\nabla u_0\|_{L^2}}{2} <1 $, so that theorem \ref{MT} and lemma \ref{Hardy} can be applied to give
$$
\bigg\|\frac{e^{\alpha| v+v_0|^2}-1}{| x|^{\frac{(b+1)(1+\varepsilon)}{1-\varepsilon}}}\bigg\|_{L_T^{\infty}(L^1)}^{\frac{1-\varepsilon}{2(1+\varepsilon)}} \lesssim \Arrowvert v+v_0\Arrowvert_{T}^{\frac{1-\varepsilon}{1+\varepsilon}}.
$$
The term $\bigg\| e^{\alpha\| v+v_0\|_{L^{\infty}_x}^{2}}-1\bigg\|_{L_T^{1}}^{\frac{1+3\varepsilon}{2(1+\varepsilon)}}$ was treated above. For ${\mathbf I}_2$  we write
\begin{eqnarray*}
{\mathbf I}_2 &\leqslant& \|\nabla (v+v_0)\|_{L_T^{\infty}(L^2_x)}\bigg\|\frac{e^{\alpha(1+\varepsilon)| v+v_0|^2}-1}{| x|^{b}}\bigg\|_{L_T^{\frac{2(1+\varepsilon)}{1+3\varepsilon}}(L^{\frac{2(1+\varepsilon)}{1-\varepsilon}}_x)}\\
&\leqslant& \|\nabla (v+v_0)\|_{L_T^{\infty}(L^2_x)} \bigg\|\frac{e^{\alpha(1+\varepsilon)| v+v_0|^2}-1}{| x|^{\frac{2b(1+\varepsilon)}{1-\varepsilon}}}\bigg\|_{L_T^{\infty}(L^1_x)}^{\frac{1-\varepsilon}{2(1+\varepsilon)}}\bigg\| e^{\alpha(1+\varepsilon)\| v+v_0\|_{L^{\infty}_x}^{2}}-1\bigg\|_{L_T^{1}}^{\frac{1+3\varepsilon}{2(1+\varepsilon)}}.
\end{eqnarray*}
Since one can take $\varepsilon>0$ such that $\frac{2b(1+\varepsilon)}{1-\varepsilon}<2$, a similar reasoning as previously allow us to get
$$
\bigg\|\frac{e^{\alpha(1+\varepsilon)| v+v_0|^2}-1}{| x|^{\frac{2b(1+\varepsilon)}{1-\varepsilon}}}\bigg\|_{L_T^{\infty}(L^1_x)}^{\frac{1-\varepsilon}{2(1+\varepsilon)}} \lesssim \Arrowvert v+v_0\Arrowvert_{T}^{\frac{1-\varepsilon}{1+\varepsilon}}.
$$
Write
\begin{eqnarray*}
\bigg\|e^{ \alpha (1+\varepsilon)\Arrowvert v+v_0\Arrowvert_{L^{\infty}_x}^2}-1\bigg\|_{L^1_T}&=& \int_{ \{t\in [0,T]\, ; \, \Arrowvert v+v_0\Arrowvert_{L^{\infty}_x} \leqslant 1\}} \left(e^{ \alpha (1+\varepsilon)\Arrowvert v+v_0\Arrowvert_{L^{\infty}_x}^2}-1\right) dt\\
&+&\int_{ \{t\in [0,T]\, ; \, \Arrowvert v+v_0\Arrowvert_{L^{\infty}_x}  > 1\}} \left(e^{ \alpha (1+\varepsilon)\Arrowvert v+v_0\Arrowvert_{L^{\infty}_x}^2}-1\right)dt.
\end{eqnarray*}

 It can easily be shown that
$$
\int_{ \{t\in  [0,T]\, ; \, \Arrowvert v+v_0\Arrowvert_{L^{\infty}_x} \leqslant 1\}}\left(e^{ \alpha (1+\varepsilon) \Arrowvert v+v_0\Arrowvert_{L^{\infty}_x}^2}-1\right) dt \leqslant C(b, \varepsilon) T^{\frac{1}{2}} \Arrowvert v+v_0\Arrowvert_{T}^{2}.
$$
Using the log estimate one can find a constant $0<\gamma<1$ such that
 $$
\int_{ \{t\in [0,T]\, ; \, \Arrowvert v+v_0\Arrowvert_{L^{\infty}_x} > 1\}} \left(e^{ \alpha (1+\varepsilon)\Arrowvert v+v_0\Arrowvert_{L^{\infty}_x}^2}-1\right) dt \lesssim  \int_{ \{t\in [0,T]\, ; \, \| u_{\omega}(t,\cdot)\|_{L^{\infty}_x} > 1\}} \Arrowvert v+v_0\Arrowvert_{C^{\frac{1}{2}}_x}^{2 (2-b) \gamma}.
$$
Hence
$$
\int_{ \{t\in [0,T]\, ; \, \Arrowvert v+v_0\Arrowvert_{L^{\infty}_x}  > 1\}} \left(e^{ \alpha (1+\varepsilon)\Arrowvert v+v_0\Arrowvert_{L^{\infty}_x}^2}-1\right) dt \lesssim T^{\frac{b}{2}} \Arrowvert v+v_0\Arrowvert_{T}^{ 2(2-b)}.
$$
We come to
$$
 {\mathbf I}_2\lesssim \Arrowvert v+v_0\Arrowvert_{T}^{\frac{2}{1+\varepsilon}} \left(T^{\frac{1}{2}} \Arrowvert v+v_0\Arrowvert_{L^{\infty}_x}^{2}+T^{\frac{b}{2}} \Arrowvert v+v_0\Arrowvert_{T}^{2(2-b)}\right)^{\frac{1+3\varepsilon}{2(1+\varepsilon)}}.
 $$
Hence
\begin{equation}
\label{estimate2}
\|\nabla f(x,v+v_0)\|_{L_T^{\frac{2(1+\varepsilon)}{3\varepsilon+1}}(L^{1+\varepsilon}_x)} \lesssim \Arrowvert v+v_0\Arrowvert_{T}^{\frac{2}{1+\varepsilon}} \left(T^{\frac{1}{2}} \Arrowvert v+v_0\Arrowvert_{L^{\infty}_x}^{2}+T^{\frac{b}{2}} \Arrowvert v+v_0\Arrowvert_{T}^{2(2-b)}\right)^{\frac{1+3\varepsilon}{2(1+\varepsilon)}}.
\end{equation}
The estimates \eqref{estimate1} and \eqref{estimate2} allow us to conclude that $\Phi$ maps $B_{T}(R)$ into itself for $0<R<\frac{1-\|\nabla u_0\|_{L^2}}{2}$ and $T>0$ sufficiently small.\\
The fact that $\Phi$ is a contraction can be carried out in a similar manner, we omit the details.
\end{proof}


\subsection{Unconditional uniqueness}
This subsection is devoted to prove an unconditional uniqueness result. Note that uniqueness in $B_R(T)$ follows from the contraction argument. In what follows, we are going to obtain the stronger statement that uniqueness hold in the natural space $C([0,T),H^1(\R^2))$.
\begin{theorem}
\label{Uniq1}
Let $T>0$ and $u_0\in H^1(\R^2)$ such
that $\|\nabla u_0\|_{L^2}<1$. Then, the Cauchy problem \eqref{eq1} has at most one solution in the space $C([0,T),H^1(\R^2))$.
\end{theorem}
The proof of Theorem \ref{Uniq1} follows immediately from the following Lemma
\begin{lemma}
\label{Uniq2}
Let $T,\; \delta>0$ be positive real numbers and $u_0\in H^1(\R^2)$ such
that $\|\nabla u_0\|_{L^2}<1$. If $u\in {\mathcal C}([0,T],
H^1(\R^2))$ is a solution of \eqref{eq1} on $[0,T]$,
then there exists a time $0<T_\delta\leq T$ such that
$u\;;\;\nabla u\in L^4([0,T_\delta], L^4(\R^2))$ and

$$
\| u\|_{L^4([0,T_\delta]\times \R^2)}+ \|\nabla
u\|_{L^4([0,T_\delta]\times \R^2)}\leq\delta.
$$
\end{lemma}
\begin{proof}

Denote by $V:=u-v_0$ with $v_0:=e^{it\Delta}u_0$. Note that $V$
satisfies

\begin{eqnarray}
\nonumber
i\partial_t V+\Delta V= f(x,V+v_0).
\end{eqnarray}
From Strichartz inequalities, to prove that $V$ and
$\nabla V$ are in $L^4_{t,x}$, it is sufficient to estimate
$\nabla^j\Big[f(x,V+v_0)\Big]$ in
the dual Strichartz norm $\|\cdot\|_{L^{\frac{2(1+\varepsilon)}{1+3\varepsilon}}_T(L^{1+\varepsilon}_x)}$ with $j=0,1$.\\
Take $\varepsilon>0$ (to be chosen later suitably). By continuity of $t\rightarrow V(t,\cdot)$, one can choose a
time $0<T_1\leq T$ such that

\begin{equation}
\label{small}
\sup_{[0,T_1]}\|V(t,\cdot)\|_{H^1}\leq\varepsilon.
\end{equation}
Observe that

$$
|V+v_0|^2\leq a |v_0|^2 + \frac{a}{a-1} |V|^2; \qquad a>1,
$$

$$
e^{x+y}-1=(e^x -1)(e^y -1)+(e^x -1)+(e^y -1); \qquad x,y \in \R,
$$
and
$$
xe^{x}\leq \frac{e^{(1+\varepsilon)x}-1}{\varepsilon}; \qquad x\geqslant0 \, , \, \varepsilon>0.
$$
In the sequel, we will only estimate the term with derivative, the other case is easier.

Applying estimate \eqref{Loc2} with $u=V+v_0$ and $v=0$, one gets
\begin{eqnarray*}
|\nabla \left(f(x,V+v_0)\right)|&\leq& C_{\varepsilon, b } \bigg\{ \frac{|V+v_0|}{|x|^{b+1}}\bigg({\rm e}^{\alpha(1+\varepsilon)|V+v_0|^2}-1\bigg)\\
\nonumber &+&   \frac{|\nabla (V+v_0)|}{|x|^{b}}\left({\rm e}^{\alpha(1+\varepsilon)|V+v_0|^2}-1\right)\bigg\}.\\
\nonumber
\end{eqnarray*}
The terms figuring in the right hand side of the last inequality are treated in exactly the same manner. We will only deal with the second one. Taking into account the above observations, we get, for $a>1$ to be chosen later conveniently
\begin{eqnarray}
\label{Part1}
\nonumber
\frac{|\nabla (V+v_0)|}{|x|^{b}}\left({\rm e}^{\alpha(1+\varepsilon)|V+v_0|^2}-1\right)&\lesssim&
|\nabla(V+v_0)|\bigg(\frac {e^{\alpha(1+\varepsilon)a|v_0|^2}-1}{|x|^{b}}\bigg)\bigg(e^{\alpha(1+\varepsilon)\frac{a}{a-1}|V|^2}-1\bigg)\\&+&
|\nabla(V+v_0)|\bigg(\frac {e^{\alpha(1+\varepsilon)a|v_0|^2}-1}{|x|^{b}}\bigg)\\&+& |\nabla(V+v_0)|\bigg(\frac {e^{\alpha(1+\varepsilon)\frac{a}{a-1}|V|^2}-1}{|x|^{b}}\bigg).
\nonumber
\end{eqnarray}
We will estimate the two first terms of the RHS of inequality \eqref{Part1}, the third one can be treated analogously. For the first term we have

\begin{equation*}
\bigg\||\nabla(V+v_0)|\bigg(\frac {e^{\alpha(1+\varepsilon)a|v_0|^2}-1}{|x|^{b}}\bigg)\bigg(e^{\alpha(1+\varepsilon)\frac{a}{a-1}|V|^2}-1\bigg)\bigg\|_{L^{\frac{2(1+\varepsilon)}{1+3\varepsilon}}_{T_1}(L^{1+\varepsilon}_x)}
\end{equation*}
\begin{equation*}
\leqslant \, \|\nabla (V+v_0)\|_{L_{T_1}^{\infty}
(L^2_x)} \bigg\|\bigg(\frac {e^{\alpha(1+\varepsilon)a|v_0|^2}-1}{|x|^{b}}\bigg)\bigg(e^{\alpha(1+\varepsilon)\frac{a}{a-1}|V|^2}-1\bigg)\bigg\|_{L_{T_1}^{\frac{2(1+\varepsilon)}{1+3\varepsilon}}(L^{\frac{2(1+\varepsilon)}{1-\varepsilon}}_x)}
\end{equation*}
\begin{equation*}
\leqslant \, \|\nabla (V+v_0)\|_{L_{T_1}^{\infty}
(L^2_x)} \bigg\|\frac {e^{\alpha(1+\varepsilon)a|v_0|^2}-1}{|x|^{b}}\bigg\|_{L_{T_1}^{\frac{2(1+\varepsilon)}{1+3\varepsilon}}(L^{\frac{4(1+\varepsilon)}{2-3\varepsilon}}_x)}
\bigg\|e^{\alpha(1+\varepsilon)\frac{a}{a-1}|V|^2}-1\bigg\|_{L_{T_1}^{\infty}(L^{\frac{4(1+\varepsilon)}{\varepsilon}}_x)}.
\end{equation*}
 Applying the H\"older inequality we obtain
 $$
 \bigg\|\frac {e^{\alpha(1+\varepsilon)a|v_0|^2}-1}{|x|^{b}}\bigg\|_{L_{T_1}^{\frac{2(1+\varepsilon)}{1+3\varepsilon}}(L^{\frac{4(1+\varepsilon)}{2-3\varepsilon}}_x)} \leqslant \bigg\|\frac{e^{\alpha(1+\varepsilon)a| v_0|^2}-1}{| x|^{\frac{4b(1+\varepsilon)}{2-3\varepsilon}}}\bigg\|_{L_{T_1}^{\infty}(L^1_x)}^{\frac{2-3\varepsilon}{4(1+\varepsilon)}}\bigg\| e^{\alpha(1+\varepsilon)a\| v_0\|_{L^{\infty}}^{2}}\bigg\|_{L_{T_1}^{\frac{2+7\varepsilon}{2(1+3\varepsilon)}}}^{\frac{2+7\varepsilon}{4(1+\varepsilon)}}.
$$

For the second term of the RHS of \eqref{Part1} one gets
\begin{equation*}
\bigg\||\nabla(V+v_0)|\bigg(\frac {e^{\alpha(1+\varepsilon)a|v_0|^2}-1}{|x|^{b}}\bigg)\bigg\|_{L^{\frac{2(1+\varepsilon)}{1+3\varepsilon}}_{T_1}(L^{1+\varepsilon}_x)}\leqslant \|\nabla (V+v_0)\|_{L_{T_1}^{\infty}
(L^2_x)} \bigg\|\frac {e^{\alpha(1+\varepsilon)a|v_0|^2}-1}{|x|^{b}}\bigg\|_{L^{\frac{2(1+\varepsilon)}{1+3\varepsilon}}_{T_1}(L^{\frac{2(1+\varepsilon)}{1-\varepsilon}}_x)}.
\end{equation*}

Therefore, we need to estimate the following three terms:

$$
{\mathcal J}_1(t):=\bigg\|\frac{e^{\alpha(1+\varepsilon)a| v_0|^2}-1}{| x|^{\frac{4b(1+\varepsilon)}{2-3\varepsilon}}}\bigg\|_{L_{T_1}^{\infty}(L^1_x)},
$$

$$
{\mathcal J}_2(t):=\bigg\|\frac {e^{\alpha(1+\varepsilon)a|v_0|^2}-1}{|x|^{b}}\bigg\|_{L^{\frac{2(1+\varepsilon)}{1+3\varepsilon}}_{T_1}(L^{\frac{2(1+\varepsilon)}{1-\varepsilon}}_x)},
$$
and

$$
{\mathcal J}_3(t):=\bigg\|e^{\alpha(1+\varepsilon)\frac{a}{a-1}|V|^2}-1\bigg\|_{L_{T_1}^{\infty}(L^{\frac{4(1+\varepsilon)}{\varepsilon}}_x)}.
$$
Take $a>1$ such that $a \|\nabla u_0\|_{L^2}<1$. This is possible since $\|\nabla u_0\|_{L^2}<1$. On the other hand, since $\frac{4b(1+\varepsilon)}{2-3\varepsilon}\rightarrow 2b$ as $\varepsilon\rightarrow0^{+}$, one can find $\varepsilon>0$ such that $\frac{4b(1+\varepsilon)}{2-3\varepsilon}<2$. With these choices for $a$ and $\varepsilon$ we can apply theorem \ref{MT} and lemma \ref{hardy} successively to obtain
$$
\int_{\R^2} \frac{e^{\alpha(1+\varepsilon)a| v_0|^2(x)}-1}{| x|^{\frac{4b(1+\varepsilon)}{2-3\varepsilon}}} \, dx  \leqslant C(a,b,\varepsilon) \|v_0\|_{H^1}^2.
$$
The same method applies for $\mathcal{J}_2$. For $\mathcal{J}_3$, taking into consideration \eqref{small},  the classical Moser-Trudinger inequality applies for suitable $a$ and $\varepsilon$ and yields
$$
\mathcal{J}_3 \leqslant C(a,\varepsilon).
$$
Therefore
$$
\bigg\||\nabla(V+v_0)|\bigg(\frac {e^{\alpha(1+\varepsilon)a|v_0|^2}-1}{|x|^{b}}\bigg)\bigg(e^{\alpha(1+\varepsilon)\frac{a}{a-1}|V|^2}-1\bigg)\bigg\|_{L^{\frac{2(1+\varepsilon)}{1+3\varepsilon}}_{T_1}(L^{1+\varepsilon}_x)} \leqslant C(a, b, \varepsilon, T_1),
$$
with $C(a, b, \varepsilon, S) \, \underset{S \rightarrow 0}{\longrightarrow}\, 0$. Hence choosing $T_1$ small enough we derive the desired estimate.
\end{proof}

Let us mention a remark about the time of local existence.

\begin{remark}\label{rem}

The time $T$ of local existence constructed in theorem \ref{local} depends only on the size of the data $u_0$ and $b$.
However, when $\| \nabla u_0\|_{L^2}<1-\varrho,$ the lifespan depends only on $\varrho$, $b$ and $\| u_0\|_{L^2}.$
\end{remark}

\subsection{Global existence}
In this section we prove a global well-posedness result in the defocusing case when $H(u_0)\leqslant 1$ (subcritical and critical regimes).\\

We start with the subcritical case where the maximal solution $u$ of $(P)$ extends globally in time by a rather simple argument. This argument do not apply in the critical case, where the situation is more delicate. Indeed, the total energy can be concentrated in the $\| \nabla u(t)\|_{L^2}$-part of the Hamiltonian. We show that such a phenomenon can not occur.\\
\subsubsection{\sf The subcritical case}\quad\\
The assumption $H(u_0)<1$ implies that $\| \nabla u_0\|_{L^2}<1$. Hence, by a fixed point argument as in the proof of local existence, the problem $(P)$ has a unique maximal solution $u$ on the time interval $[0, T^*)$ where $0<T^*\leqslant \infty$  is the maximal lifespan.\\
We argue by contradiction. Assume that $T^*$ is finite. Then, we have
$$\sup_{t\in[0,T^*)}\| \nabla u(t)\|_{L^2}\leqslant H(u_0)<1.$$
Consider for $0<t_0<T^*$ the following Cauchy problem:
$$\left \{\begin{array}{l}
i\partial_t v+\Delta v = f(x,v),\\
v(t_0) = u(t_0)
\end {array}
 \right.$$\\
By the local existence theory, we can see that there exists a non-negative $\tau$ and a unique solution $v$ to our problem on the interval $[t_0,T]$ where $T=t_0+\tau$. Using Remark \ref{rem} and the conservation laws, we see that $\tau$ depends only on $\rho:=\frac{1-H(u_0)}{2}$, $b$ and $\|u_0\|_{L^2}$.  Choosing $t_0$ close enough to $T^*$ such that $T^*-t_0<\tau$, one can extend $u$ beyond the maximal time $T^*$. This yields to a contradiction.\\

\subsubsection{\sf The critical case}\quad\\
Let $u$ be the maximal solution of $(P)$ defined on the interval $[0,T^*)$, where $0<T^*\leqslant \infty$   is the lifespan of $u$. Since $H(u_0)=1$, a concentration phenomena can occur, that is
$$
\limsup_{t\longrightarrow T^*}\| \nabla u(t)\|_{L^2}=1.
$$
Therefore, the argument used above in the subcritical case does not apply. We show here that such a concentration phenomena does not happen.
Assume that $T^*$ is finite and let us derive a contradiction. Before doing so, We state and prove the following proposition.
\begin{proposition}
\label{limits}
The maximal solution u satisfies:
\begin{itemize}
 \item [$(i)$] $ \displaystyle\limsup_{t\longrightarrow T^* }\| \nabla u(t)\|_{L^2} =1$.
 \item [$(ii)$] There exists a sequence of times $\{t_n\}_{n=1}^{\infty}$ in $[0,T^*]$ converging to $T^* $such that
 $$
 \int_{\R^2}\frac{| u(t_n,x)|^4}{| x|^a}dx  \underset{n\to \infty}{\longrightarrow}0.
 $$
 \end{itemize}
\end{proposition}
\begin{proof}
Recall that  $$H(u(t)):=  \int_{\R^2}|\nabla u(t)|^2dx+ \frac{1}{\alpha} \int_{\R^2}\left( (e^{\alpha| u(t)|^2} -1-\alpha | u(t)|^2) \right)\frac{dx}{| x|^a}.$$
Hence, for all $0\leqslant t<T^*$ we have $$\|\nabla u(t)\|_{L^2}^2\leqslant H(u(t))=1.$$
 Therefore
 $$\limsup_{t\longrightarrow T^* }\|\nabla u(t)\|_{L^2}\leqslant1.$$
Set $L:=\limsup_{t\longrightarrow T^* }\|\nabla u(t)\|_{L^2}$ and assume $L<1$.\\
The definition of the superior limit ensures the existence of a sequence of times $\{t_n\}_{n=1}^{\infty}$ in $[0,T^*]$ converging to $T^*$ such that
$$
\|\nabla u(t_n)\|_{L^2}  \underset{n\to \infty}{\longrightarrow}L.
$$
Thus, there exists $N=N(L) \in \N ^*$  such that, for all $n\geqslant N$ we have
$$
\|\nabla u(t_n)\|_{L^2} \leqslant L+\frac{1-L}{2}=1-\frac{1-L}{2}.
$$
Let $n\geqslant N$. From the local theory, one can construct a non-negative $\tau$ (depending only on $\frac{1-L}{2}$, $b$ and $\|u_0\|_{L^2}$) and a unique solution
$v$ on $[t_n,t_n+\tau]$ of the problem $$\left \{\begin{array}{l}
i\partial_t v+\Delta v =f(x,v),\\
v(t_n,x) = u(t_n,x) .
\end {array}
 \right.$$
Hence, choosing $n\geqslant N$ such that $T^*-t_n<\tau$, we are able to extend the solution $u$ beyond $T^*$, which is absurd.\\

It remains now to establish $(ii)$. Note that, for all $x\geqslant 0$
\begin{equation}
\label{estimate3}
\frac{\alpha}{2} x^4 \leqslant  \frac{e^{\alpha x^2} -1}{\alpha}  - x^2.
\end{equation}
Applying \eqref{estimate3} with $x=|u|$, integrating over $\R^2$ and using the fact that $\frac{2}{\alpha}\leqslant 1$, we obtain
\begin{equation}
\label{estimate4}
\int_{\R^2} \frac{| u(t,x)|^4}{| x|^b} \,dx+\| \nabla u(t)\|_{L^2}^2 \leqslant 1, \qquad t \in [0,T^*).
\end{equation}
From $(i)$ we know that there exists a sequence of times $\{t_n\}_{n=1}^{\infty}$ in $[0,T^*]$ converging to $T^*$ such that
$$
\|\nabla u(t_n)\|_{L^2}  \underset{n\to \infty}{\longrightarrow}1.
$$
By plugging this sequence into \eqref{estimate4}, one achieves the proof of the second assertion.
 \end{proof}

Now we are in position to prove the global existence in the critical regime. Let $\{t_n\}_{n=1}^{\infty}$ be the sequence of times given by the second point of  proposition \ref{limits}. Lemma \ref{LG} and the Cauchy-Schwarz inequality yield
\begin{eqnarray*}
\int_{B(S)} | u_0(x)|^2 dx
&\leq&C(b,S, S') \bigg(\int_{\R^2}\frac{ |u(t_n,x)|^4}{| x|^a} \, dx\bigg)^{\frac{1}{2}}+ \frac{C(\|u_0\|_{L^2})}{S'} t_n.
\end{eqnarray*}
Taking the limit in $n$, we find that
$$\int_{B(S)} | u_0(x)|^2 dx\leqslant \frac{C(\|u_0\|_{L^2})}{S'} T^*.$$
Let $S'$ tend to $\infty$. We get
$$\int_{B(S)} | u_0(x)|^2 dx=0,$$
which is a contradiction (unless $u_0\equiv 0$). This ends the proof of Theorem \ref{EG}.






\section{Appendix}
\begin{definition}
Let $E$ be a bounded measurable set of $\mathbb{R}^{d}$. Let $u: E\rightarrow \mathbb{R}$ be a measurable function. The distribution function of $u$ is given by
$$
\mu_{u}(t):=|\{u>t\}|, \qquad t \in \R.
$$
Here, the notation $|\Omega|$ stands for the $d$-dimensional Lebesgue measure of $\Omega$.\\
The (unidimensional) decreasing rearrangement of $u$, denoted by $u^{\#}$, is defined on $[0,|E|]$ by
\[u^{\#}(s) = \left\{
\begin{array}{l l}
  \textit{ess sup}(u) & \quad s=0,\\
  \textit{inf} \, {\{t; \, \mu_{u}(t)<s \}}& \quad s>0.\\ \end{array} \right. \]
  Here "$\textit{ess sup}(u)$" is the essential supremum of $u$.
 \end{definition}
For the rest of this section, we denote by $E^*$ the open ball centered at the origin and having the same measure as $E$ and by $\omega_d$ the volume of the unit ball in $\mathbb{R}^{d}$.
\begin{definition}
Let $E$ be a bounded domain of $\mathbb{R}^{d}$ and $u: E\rightarrow \mathbb{R}$ a measurable function. The Schwarz symmetrization (or the spherically symmetric and decreasing rearrangement) of $u$ is the function $u^*: E^*\rightarrow \mathbb{R}$ defined by
$$
u^*(x)=u^{\#}(\omega_d |x|^d), \qquad x \in E^*.
$$
\end{definition}
The Schwarz symmetrization $u^*$ of $u$ enjoys several interesting properties:\\
\begin{itemize}
\item It is a radially symmetric and decreasing function.\\
\item $u$ and $u^*$ are equi-measurable. That is, they have the same distribution function.\\
\item It has the same $L^p$-norm as $u$. More generally, if $F: \mathbb{R}\rightarrow \mathbb{R}$ is a Borel measurable function such that either $F\geqslant0$ or $F(u) \in L^1(E)$, then
$$
\int_{E^*}F(u^*)(x) \, dx=\int_{E}F(u)(x) \, dx.
$$
\item If $F: \mathbb{R}\rightarrow \mathbb{R}$ is a non-decreasing function, then $(F(u))^*=F(u^*)$.\\
\end{itemize}
The following two theorems are relevant for us.
\begin{theorem} [\sf Hardy-Littlewood] \quad \\
Let $f \in L^p(E)$ and $g \in L^q(E)$ with $1\leqslant p, q \leqslant \infty$ such that $\frac{1}{p}+\frac{1}{q}=1$. Then
$$
\int_{E} f(x) g(x) \, dx \leqslant \int_{E^*} f^*(x) g^*(x) \, dx.
$$
\end{theorem}
\begin{theorem}[\sf Polya-Szeg\"{o}] \quad \\
Let $1\leqslant p < \infty$. Let $E$ be a bounded domain of $\mathbb{R}^{d}$ and $u \in W^{1,p}_{0}(E)$ such that $u \geqslant 0 $. Then
$$
\int_{E^*} |\nabla u^*|^p \, dx \leqslant \int_{E} |\nabla u|^p \, dx.
$$
In particular, $u^* \in W^{1,p}_{0}(E^*)$.
\end{theorem}

For proofs of these results as well as the properties mentioned above the reader is referred to the book of S. Kesaven \cite{Kesaven} (chapters 1 and 2).\\

Now we state and a prove a result used in the proof of lemma \ref{hardy}.
\begin{lemma} [\sf Strauss radial lemma]\quad\\
Let $2\leqslant p<\infty$. There exists $C_p>0$ such that for all $u\in H^1_{\textit{rad}}(\R^2)$,
\begin{eqnarray}\label{rttt}
| u(x)|\leqslant \frac{C_p}{r^{\frac{2}{2+p}}}\| u\|_{H^1(\R^2)},
\end{eqnarray}
where $r=| x|$.
\end{lemma}
\begin{proof}
By density, it suffices to consider smooth compactly supported functions. Let $u(x)=\psi(| x|)$, $\psi \in \mathcal{D}([0,\infty[)$. We have
$$
\psi^{\frac{p+2}{2}}(r)=-\frac{p+2}{2} \int_{0}^{+\infty} \psi'(s) \psi(r)^{\frac{p}{2}} \, ds.
$$
Thus
$$
|\psi|^{\frac{p+2}{2}}(r)\leqslant \frac{p+2}{2r} \int_{0}^{+\infty} |\psi'|(s)|\psi(r)|^{\frac{p}{2}} s \, ds\leqslant   \frac{p+2}{2r} \|\nabla u\|_{L^2(\R^2)}\|u\|_{L^p(\R^2)}^{\frac{p}{2}}.
$$
We conclude using the Sobolev embedding $H^1(\R^2)\hookrightarrow L^p(\R^2)$.
\end{proof}

\begin{theorem}[\sf Poincar\'e inequality]\quad\\
Let $\Omega$ be a bounded domain of $\R^d$. Let $1\leqslant p <\infty$. Then, a constant $C$ exists depending only on $\Omega$ and $p$ such that, for all $u \in W^{1,p}_0(\Omega)$, the following inequality holds
$$
\|u\|_{L^p(\Omega)} \leqslant C \|\nabla u\|_{L^p(\Omega)}.
$$
\end{theorem}
The Poincar\'e inequality implies the following.
\begin{corollary}
The norms $\|\cdot\|_{W^{1,p}(\Omega)}$ and $\|\nabla \cdot\|_{L^p(\Omega)}$ are equivalent.
\end{corollary}
\begin{theorem}[\sf A characterization of $W^{1,p}_0(\Omega)$]\quad\\
Suppose $\Omega$, a bounded open set of $\R^d$, is of class $C^1$. Let $u \in L^p(\Omega)$ with $1<p<\infty$. The following properties are equivalent:
\begin{itemize}
\item [$(i)$] $u \in W^{1,p}_0(\Omega)$;
\item [$(ii)$] there exists a constant $C$ such that, for all $\phi \in \mathcal{D}(\R^d)$ and all $j\in \{1,\cdots,d\}$
$$
\bigg|\int_{\Omega} u \, \partial_{x_j}\phi\bigg| \leqslant C \|\phi\|_{L^{p'}(\Omega)}$$,\\
\item [$(iii)$] the function
\[ \bar{u}(x) := \left\{
\begin{array}{l l}
  u(x) & \quad \text{if} \quad x \in \Omega,\\
  0 & \quad \text{if} \quad x \in \R^d \backslash \Omega,\\ \end{array} \right. \]
belongs to $W^{1,p}(\R^d)$.

\end{itemize}
\end{theorem}
\begin{proof}
See \cite{Brezis}.
\end{proof}

\end{document}